\documentclass[a4paper,11pt]{amsart}

\usepackage[utf8]{inputenc}
\usepackage{microtype}
\usepackage{amsfonts}
\usepackage{amsmath}
\usepackage{float}
\usepackage[dvipsnames]{xcolor}
\usepackage{hyperref}
\usepackage[margin=1.25in]{geometry}

\usepackage{amssymb}
\usepackage{enumitem}
\usepackage{mathrsfs}

\usepackage{empheq}

\newtheorem{theorem}{Theorem}[section]
\newtheorem{lemma}[theorem]{Lemma}
\newtheorem{corollary}[theorem]{Corollary}

\theoremstyle{definition}
\newtheorem{definition}[theorem]{Definition}

\newcommand{\N}{\mathbb{N}}
\newcommand{\Z}{\mathbb{Z}}
\newcommand{\R}{\mathbb{R}}

\DeclareMathOperator{\diam}{diam}
\DeclareMathOperator{\Rips}{\mathcal{VR}}

\newcommand{\lk}{\operatorname{lk}}
\newcommand{\st}{\operatorname{st}}
\newcommand{\dst}{\st^\downarrow\!}
\newcommand{\dlk}{\lk^\downarrow\!}
\newcommand{\dflk}{\lk^\downarrow_\partial}
\newcommand{\dclk}{\lk^\downarrow_\delta}

\DeclareMathOperator{\F}{F}

\numberwithin{equation}{section}

\begin{document}

\title{Contractible Vietoris--Rips complexes of $\mathbb{Z}^n$}
\date{\today}


\author[M.~C.~B.~Zaremsky]{Matthew C.~B.~Zaremsky}
\address{Department of Mathematics and Statistics, University at Albany (SUNY), Albany, NY}
\email{mzaremsky@albany.edu}

\begin{abstract}
We give a new, short proof of a result of Virk, that the Vietoris--Rips complex of the group $\Z^n$ with the standard word metric is contractible at large enough scales. This is inspired by a key observation in Virk's proof, but we use Bestvina--Brady discrete Morse theory to get a very short proof with better bounds. In the course of this, we get a new, general criterion for a metric space to have contractible Vietoris--Rips complexes at large enough scales, which could prove useful in the future.
\end{abstract}

\maketitle
\thispagestyle{empty}

\section{Introduction}

For a metric space $X$ and $t\in\R\cup\{\infty\}$, the \emph{Vietoris--Rips complex} $\Rips_t(X)$ is the simplicial flag complex with vertex set $X$ such that vertices span a simplex whenever they are pairwise within distance $t$ of each other. It is an interesting problem in geometric group theory to try and determine which finitely generated groups $G$, viewed as metric spaces via word metrics coming from finite generating sets, admit contractible $\Rips_t(G)$ for $t<\infty$. For example, if $G$ is hyperbolic then this holds by a result of Rips, as explained in \cite[III.$\Gamma$.3.23]{bridson99}. Having a contractible Vietoris--Rips complex has several implications, for example the group is consequently finite presented, and even of type $\F_\infty$; see \cite{zaremsky22} for more in this direction. Beyond the hyperbolic case, contractibility is quite hard to prove. For the very basic group $\Z^n$, with the standard word metric, the contractibility question was posed by the author in the paper \cite{zaremsky22}, the first version of which is from 2018, but despite much effort by many people over the years, a proof remained elusive until just recently. In \cite{virk25}, \v Ziga Virk finally proved that $\Rips_t(\Z^n)$ is contractible for all suitably large $t$.

The primary purpose of this note is to give a new, short proof of Virk's result, using Bestvina--Brady discrete Morse theory. This is a topological tool for analyzing simplicial complexes, first introduced by Bestvina and Brady in \cite{bestvina97}. In \cite{zaremsky22}, the author established an approach to understanding Vietoris--Rips complexes using Bestvina--Brady Morse theory, but the sufficient conditions for contractibility given in \cite{zaremsky22} (for example see \cite[Theorem~3.5]{zaremsky22}) seem to be too difficult to verify for $\Z^n$. Here we prove a new sufficient condition, in Theorem~\ref{thrm:cible}, where the main requirement is that every finite subset with sufficiently large diameter lies in a ball of radius strictly smaller than this diameter. One of the crucial insights in \cite{virk25} is that this applies to $\Z^n$. We prove in Corollary~\ref{cor:Zn} that $\Z^n$ also satisfies the other condition required for Theorem~\ref{thrm:cible}, and so $\Rips_t(\Z^n)$ is contractible for all $t\ge n^2+n-1$. This is a slight improvement over the cubic bound in \cite{virk25}, and we conjecture that $\Rips_t(\Z^n)$ is contractible for all $t\ge n$. (It is easy to see that this bound is optimal.)

A secondary purpose of this note is to give a short summary of the machinery from \cite{zaremsky22} for applying Bestvina--Brady Morse theory to Vietoris--Rips complexes. The paper \cite{zaremsky22} is quite long and technical, but the basic idea is simple enough that we can explain it in about a page and half here, in Section~\ref{sec:background}.

\subsection*{Acknowledgments} Many thanks to \v Ziga Virk for explaining his proof in \cite{virk25} to me. I am also grateful to the many, many people who I talked to over the years trying to prove this result, especially Henry Adams and Brendan Mallery. Thanks are also due to Henry Adams, Gabriel Minian, and Carl Ye for pointing out some errors in a previous version, and to the referee for many helpful suggestions.

\section{Vietoris--Rips complexes and Bestvina--Brady discrete Morse theory}\label{sec:background}

It will be convenient for us to work with the barycentric subdivision of $\Rips_t(X)$, that is, the geometric realization of the poset of its simplices, which is homeomorphic to $\Rips_t(X)$. For simplicity we will not change the notation, so from now on when we write $\Rips_t(X)$, we mean the geometric realization of the poset of all finite non-empty subsets $S\subseteq X$ with $\diam(S)\le t$. Thus, a $k$-simplex in $\Rips_t(X)$ is a chain $S_0\subsetneq \cdots\subsetneq S_k$ of finite non-empty subsets of $X$ with $\diam(S_i)\le t$ for all $i$. Note that $S$ now denotes both a subset of $X$ and also the corresponding vertex of $\Rips_t(X)$, and both viewpoints will be important; we will be careful to avoid potential confusion.

In \cite{zaremsky22}, the author set up a general version of Bestvina--Brady-style discrete Morse theory, for applications to Vietoris--Rips complexes. More precisely, see \cite[Definition~1.1]{zaremsky22} for the general definition of Morse function and \cite[Lemma~1.8]{zaremsky22} for the general Morse Lemma. Here we are only concerned with certain metric spaces, so we can afford to set up an easier version of discrete Morse theory.

\begin{definition}[Morse function]\label{def:morse}
Let $K$ be a simplicial complex. A map $h\colon K\to\R$ is a \emph{(discrete) Morse function} if the image $h(K^{(0)})$ is a discrete, closed subset of $\R$, and the restriction of $h$ to any simplex takes distinct values on the vertices of the simplex.
\end{definition}

In the original Bestvina--Brady paper \cite{bestvina97}, $K$ can be any affine cell complex, not just a simplicial complex, and the last condition is instead that $h$ is non-constant on edges and affine on cells. In \cite[Definition~1.1]{zaremsky22}, this is further weakened, essentially to not require $h$ to be discrete on vertices, but rather to just not have any $h$-values of vertices accumulate downward. The phrasing is quite different in \cite{zaremsky22}, but for our purposes here, Definition~\ref{def:morse} is enough.

\begin{definition}[Star/link, descending star/link]
Let $K$ be a simplicial complex. For a vertex $v$ of $K$, the \emph{star} of $v$ is the subcomplex $\st(v)$ of $K$ consisting of all simplices containing $v$ together with their faces. The \emph{link} of $v$ is the subcomplex $\lk(v)$ of $\st(v)$ consisting of all simplices in the star not containing $v$. Now let $h\colon K\to\R$ be a Morse function. The \emph{descending star} of $v$ is the subcomplex $\dst(v)$ of $\st(v)$ consisting of all simplices containing $v$ as their vertex with maximum $h$ value together with their faces. The \emph{descending link} of $v$ is the subcomplex $\dlk(v)$ of $\dst(v)$ consisting of all simplices in the descending star not containing $v$.
\end{definition}

Given a Morse function $h\colon K\to \R$ and a value $t\in\R$, let $K^{h\le t}$ be the full subcomplex of $K$ spanned by all vertices $v$ with $h(v)\le t$.

\begin{lemma}[Morse Lemma]\label{lem:morse}
Let $h\colon K\to\R$ be a Morse function and let $s<t$ in $\R$. If $\dlk(v)$ is contractible for all vertices $v$ with $s<h(v)\le t$, then the inclusion $K^{h\le s}\to K^{h\le t}$ is a homotopy equivalence. If $\dlk(v)$ is contractible for all vertices $v$ with $h(v)>t$, then the inclusion $K^{h\le t}\to K$ is a homotopy equivalence.
\end{lemma}

\begin{proof}
By \cite[Corollary~1.11]{zaremsky22}, the inclusion induces an isomorphism in all homotopy groups, so the result is immediate from the Whitehead theorem.
\end{proof}

\medskip

Now let us discuss how to apply Bestvina--Brady Morse theory to Vietoris--Rips complexes. Let $X$ be a metric space. From now on we will assume the following property (*), which in particular holds whenever $X$ is a finitely generated group with a word metric coming from a finite generating set.

\medskip

\noindent\textbf{(*)}: All distances in $X$ are integers, and for all $t\in\N$ there exists $n_t\in\N$ such that $|S|\le n_t$ for all $S\subseteq X$ with $\diam(S)=t$.

\medskip

In all that follows, we fix some choice of $n_t$ as in (*), for each $t\in\N$.

The complex $\Rips_\infty(X)$ is the geometric realization of the poset of finite non-empty subsets of $X$. This poset is directed, so $\Rips_\infty(X)$ is contractible. Let $h\colon \Rips_\infty(X) \to \R$ be the map that sends $S$, viewed as a vertex of $\Rips_\infty(X)$, to
\[
h(S)\coloneqq \diam(S) - \frac{|S|-1}{n_{\diam(S)}}\text{,}
\]
where the measurements $\diam(S)$ and $|S|$ consider $S$ as a subset of $X$. This defines $h$ on the 0-skeleton of $\Rips_\infty(X)$, and we extend it affinely to each simplex to get a map on $\Rips_\infty(X)$. Note that $h(S)\in (\diam(S)-1,\diam(S)]$ for all $S$. In particular since all distances in $X$ are integers, for any $t\in\N$ we have that $h(S)\le t$ if and only if $\diam(S)\le t$, so $\Rips_t(X)=\Rips_\infty(X)^{h\le t}$. Note that $h$ is a Morse function as in Definition~\ref{def:morse}, since (*) ensures that the set of $h$ values of vertices is discrete and closed in $\R$, and $h$ takes distinct values on the vertices of any simplex. As a remark, we could relax the condition in (*) requiring that all distances are integers to just requiring the set of distances to be closed and discrete in $\R$, and adjust the definition of $h$ appropriately. Since our main concern here is $\Z^n$, for the sake of simplicity we will not expound on this here.

\begin{corollary}\label{cor:morse_rips}
Let $X$ and $h$ be as above and let $s<t$ in $\N$. Suppose that for all finite non-empty $S\subseteq X$ with $s<\diam(S)\le t$, the descending link $\dlk(S)$ with respect to $h$ is contractible. Then the inclusion $\Rips_s(X)\to\Rips_t(X)$ is a homotopy equivalence. Moreover, for any $t\in\N$, if for all finite non-empty $S\subseteq X$ with $\diam(S)>t$ the descending link $\dlk(S)$ with respect to $h$ is contractible, then $\Rips_t(X)$ is contractible.
\end{corollary}

\begin{proof}
We have $\Rips_t(X)=\Rips_\infty(X)^{h\le t}$ for all $t$, and $\diam(S)>t$ if and only if $h(S)>t$, so this follows from Lemma~\ref{lem:morse}.
\end{proof}

The descending link of $S$ is described as follows. First, the link of $S$ in $\Rips_\infty(X)$ is the full subcomplex spanned by all $S'$ satisfying $S'\subsetneq S$ and all $\widetilde{S}$ satisfying $\widetilde{S}\supsetneq S$. By the construction of $h$, the descending link is the full subcomplex spanned by all such $S'$ satisfying $\diam(S')<\diam(S)$ and all such $\widetilde{S}$ satisfying $\diam(\widetilde{S})=\diam(S)$. Of course $S'\subseteq \widetilde{S}$ for all $S'$ and $\widetilde{S}$ as above, so $\dlk(S)$ decomposes as the join of the full subcomplex spanned by all such $S'$ with the full subcomplex spanned by all such $\widetilde{S}$.

\begin{definition}[Descending (co)face link]
The \emph{descending face link} $\dflk(S)$ of $S$ is the full subcomplex of $\dlk(S)$ spanned by all $S'$ satisfying $S'\subsetneq S$ and $\diam(S')<\diam(S)$, and the \emph{descending coface link} $\dclk(S)$ of $S$ is the full subcomplex of $\dlk(S)$ spanned by all $\widetilde{S}$ satisfying $\widetilde{S}\supsetneq S$ and $\diam(\widetilde{S})=\diam(S)$. We have that $\dlk(S)$ is the join of $\dflk(S)$ with $\dclk(S)$.
\end{definition}

In particular, if either of $\dflk(S)$ or $\dclk(S)$ is contractible, then $\dlk(S)$ is contractible.

\section{Contractible Vietoris--Rips complexes}\label{sec:cible}

\begin{theorem}\label{thrm:cible}
Let $(X,d)$ be a metric space satisfying \textnormal{(*)}. Fix $t_0$, and suppose that for all $t>t_0$ there exists $r_t<t$ such that every $S\subseteq X$ with $\diam(S)=t$ lies in some ball of radius $r_t$. Assume moreover that for some such ball, its center lies within $t$ of the center of any other such ball. Then $\Rips_{t_0}(X)$ is contractible (as is $\Rips_t(X)$ for all $t>t_0$).
\end{theorem}

\begin{proof}
By Corollary~\ref{cor:morse_rips}, which applies since $X$ satisfies (*), it suffices to prove that $\dlk(S)$ is contractible for all $S$ with $\diam(S)=t>t_0$. First suppose that there exists $y\in S$ such that $S$ lies in the ball of radius $r_t$ centered at $y$. Now we claim that $\dflk(S)$ is contractible. Let $S'\subseteq S$ be in $\dflk(S)$, so $\diam(S')<\diam(S)$. Since $d(y,s)\le r_t<t$ for all $s\in S$, we know that $\diam(S'\cup\{y\})<\diam(S)$, so $S'\cup\{y\}$ is still in $\dflk(S)$. Thus, by standard tools, e.g., \cite[Subsection~1.3]{quillen78}, $S'\mapsto S'\cup\{y\}$ induces a homotopy equivalence from $\dflk(S)$ to the star of $\{y\}$ in $\dflk(S)$, which is contractible. Now suppose that no such $y$ can be found in $S$. By our hypothesis there does still exist $y\in X$ such that $S$ lies in the ball of radius $r_t$ centered at $y$, but now every such $y$ must lie outside $S$. We claim that $\dclk(S)$ is contractible. Let $\emptyset\ne Y\subseteq X$ be the set of all $y$ such that $S$ lies in the ball of radius $r_t$ centered at $y$. For each $y\in Y$, let $Z_y$ be the star of $S\cup\{y\}$ in $\dclk(S)$, i.e., the subcomplex of all $\widetilde{S}\supsetneq S$ with $\diam(\widetilde{S})=\diam(S)$ and $y\in\widetilde{S}$. For any $\widetilde{S}$ in $\dclk(S)$, applying the hypothesis to $\widetilde{S}$ we know it lies in a ball of radius $r_t$ centered at some element, and since $S\subseteq \widetilde{S}$, this element must lie in $Y$. This shows that the $Z_y$ cover all of $\dclk(S)$. Being a star, each $Z_y$ is contractible. We have that $Z_{y_1}\cap\cdots\cap Z_{y_k}$ is non-empty if and only if $\diam(S\cup\{y_1,\dots,y_k\})\le t$, if and only if $\diam\{y_1,\dots,y_k\}\le t$, and in this case the intersection is the star of $S\cup\{y_1,\dots,y_k\}$ in $\dclk(S)$, hence is contractible. By a standard nerve lemma, we conclude that $\dclk(S)$ is homotopy equivalent to the nerve of its covering by the $Z_y$, which is isomorphic to $\Rips_t(Y)$. By the ``moreover'' part of the hypothesis in the statement of the theorem, there exists $y_0\in Y$ such that $d(y_0,y)\le t$ for all $y\in Y$, so $\Rips_t(Y)$ is a cone on some $y_0\in Y$, hence contractible.
\end{proof}

\begin{corollary}\label{cor:Zn}
The group $\Z^n$ with the standard word metric satisfies the hypotheses of Theorem~\ref{thrm:cible} using $t_0=n^2+n-1$. Hence $\Rips_t(\Z^n)$ is contractible for all $t\ge n^2+n-1$.
\end{corollary}

\begin{proof}
Since $\Z^n$ satisfies (*), we are allowed to use Theorem~\ref{thrm:cible}. View $\Z^n$ inside the space $\R^n$. The word metric on $\Z^n$ coincides with the $\ell^1$-metric on $\R^n$, so we will always consider this metric, call it $d$. For all $t\ge n^2+n$, let
\[
r_t\coloneqq tn/(n+1) + n/2\text{.}
\]
Since $t\ge n^2+n$, we have $r_t\le t - n/2$, so in particular $r_t<t$. Now let $S\subseteq \Z^n$ with $\diam(S)=t$. We need to verify the hypotheses of Theorem~\ref{thrm:cible}. First suppose $|S|\le n+1$. Let $x_0 \in \R^n$ be the center of mass of $S$, i.e.,
\[
x_0\coloneq \frac{1}{|S|}\sum\limits_{s\in S} s\text{,}
\]
and let $y_0$ be any element of $\Z^n$ within distance at most $n/2$ of $x_0$. Then for any $s_0\in S$ we have $d(s_0,y_0)\le n/2 + \frac{1}{|S|}\sum\limits_{s\in S}d(s_0,s) \le n/2 + \frac{1}{|S|}(0+(|S|-1)t)\le n/2 + tn/(n+1) = r_t$. We conclude that $S$ lies in the ball of radius $r_t$ centered at $y_0$. Now let $y\in\Z^n$ be any other element such that $S$ lies in the ball of radius $r_t$ centered at $y$. Then $d(y,y_0)\le n/2 + \frac{1}{|S|}\sum\limits_{s\in S}d(y,s) \le n/2 + r_t \le t$. This confirms the hypotheses of Theorem~\ref{thrm:cible} for all $S$ with $|S|\le n+1$. Now suppose $|S|>n+1$. The proof of this case is inspired by the proof of Proposition~2.12 in \cite{amir85}. For each $s\in S$ let $C_s\subseteq \R^n$ be the intersection of the ball of radius $tn/(n+1)$ centered at $s$ with the convex hull of $S$. Since balls are convex, each $C_s$ is convex. For every $T\subseteq S$ with $|T|\le n+1$, by the proof of the previous case we have that the $C_t$ for all $t\in T$ have non-empty intersection in $\R^n$. Now Helly's theorem says that the $C_s$ for all $s\in S$ have non-empty intersection in $\R^n$. Say $x$ is a point in this intersection, and let $y_0$ be any element of $\Z^n$ within distance at most $n/2$ of $x$. Now $S$ lies in the ball of radius $tn/(n+1)$ centered at $x$ and hence in the ball of radius $r_t$ centered at $y_0$. Finally, let $y\in\Z^n$ be any other element such that $S$ lies in the ball of radius $r_t$ centered at $y$. The convex hull of $S$ also lies in this ball, hence $x$ lies in this ball, i.e., $d(y,x)\le r_t$, and so $d(y,y_0)\le r_t + n/2 \le t$. This confirms the hypotheses of Theorem~\ref{thrm:cible} for all $S$, and we are done.
\end{proof}

\bibliographystyle{alpha}

\end{document}